\documentclass[a4paper,12pt,twoside,reqno]{amsart}
\usepackage{amsmath}
\usepackage{amsthm}
\usepackage{amssymb}
\usepackage{amsfonts}%
\usepackage{mathrsfs}%
\theoremstyle{plain}
\newtheorem{thm}{\indent\bf Theorem}[section]
\newtheorem{lem}[thm]{\indent\bf Lemma}
\newtheorem{prop}[thm]{\indent\bf Proposition}
\newtheorem{cor}[thm]{\indent\bf Corollary}
\newtheorem*{thma}{\indent\bf Theorem A}
\newtheorem*{thmb}{\indent\bf Theorem B}
\theoremstyle{definition}
\newtheorem{rem}{\indent\it Remark}[section]

\newtheorem{defi}{\indent\bf Definition}[section]

\numberwithin{equation}{section}
\numberwithin{figure}{section}

\def \d {\mathrm{d}}

\def \re {\mathrm{Re\,}}
\def \im {\mathrm{Im\,}}

\begin{document}
\title[First Painlev\'e transcendents]
{Explicit error term of the elliptic asymptotics 
for the first Painlev\'e transcendents}

\author[Shun Shimomura]{Shun Shimomura} 

\address{Department of Mathematics, Keio University, 
3-14-1, Hiyoshi, Kohoku-ku, Yokohama 223-8522 Japan\quad
{\tt shimomur@math.keio.ac.jp}}
 \date{}
\begin{abstract}
For the first Painlev\'e transcendents Kitaev established an asymptotic
representation in terms of the Weierstrass pe-function
in cheese-like strips tending along generic directions near the point 
at infinity. 
The error term of this asymptotic expression is presented by an explicit
formula, which leads to the error bound of exponent $-1$. 
\end{abstract}
\maketitle
\allowdisplaybreaks

\section{Introduction}\label{sc1}
The first Painlev\'e equation 
\begin{equation}\label{1.1}
  y''= 6y^2+x 
\end{equation}
governs the isomonodromy deformation of the two-dimensional linear system
\begin{align}\label{1.2}
\frac{d\Psi}{d\xi}= & \Bigl( (4\xi^4 + x+ 2y^2)\sigma_3 
- i(4y \xi^2 + x+ 2y^2)\sigma_2 -(2y'\xi +\tfrac 12\xi^{-1})\sigma_1 
\Bigr)\Psi,
\\
\notag
& \sigma_1= \begin{pmatrix}  0 & 1 \\  1 & 0 \end{pmatrix},
\quad
 \sigma_2= \begin{pmatrix}  0 & -i  \\  i & 0 \end{pmatrix},
\quad
 \sigma_3= \begin{pmatrix}  1 & 0  \\  0 & -1 \end{pmatrix}
\end{align}
(\cite{JM}, \cite{Ka}, \cite{FIKN}). System \eqref{1.2} admits the canonical 
solutions $\Psi_k(\xi)$ $(k\in \mathbb{Z})$ represented by
$\Psi_k(\xi)=(I+O(\xi^{-1}))\exp((\tfrac 45 \xi^5+x\xi)\sigma_3)$
as $\xi\to \infty$ through the sector $-3\pi/{10} +\pi k/5<\arg \xi<\pi/{10}
+\pi k/5$, and the Stokes matrices
$$
S_{2l+1}=\begin{pmatrix} 1 & s_{2l+1} \\ 0 & 1 \end{pmatrix}, \quad  
S_{2l}=\begin{pmatrix} 1 & 0 \\ s_{2l}  & 1 \end{pmatrix}, \quad l=0,1,2,
\ldots
$$
are defined by $\Psi_{k+1}(\xi)=\Psi_k(\xi) S_k$.  
Each solution of \eqref{1.1} is parametrised by 
$(s_j)_{1\le j \le 5}$ on the manifold of monodromy data for \eqref{1.2}, 
which is a two-dimensional complex manifold in $\mathbb{C}^5$ given by 
$S_1S_2S_3S_4S_5=i\sigma_2$ 
(see \cite{Ka}, \cite{Ka-Ki}, \cite{K1}, \cite{K-3}).  
For a general solution thus parametrised, Kapaev \cite{Ka}
obtained asymptotic representations as $x \to \infty$ along the Stokes rays
$\arg x=\pi + 2\pi k/5$ $(k=0,\pm 1, \pm 2)$, and subsequent studies on related
connection formulas and the Stokes phenomenon are found in  
\cite{T1}, \cite{T2}, \cite{Ka-2004},
\cite{SV}, \cite{BSSV}, and \cite{L-R}, \cite{IM} for the $\tau$-function.
Along generic directions the Boutroux ansatz \cite{Boutroux} suggests the
asymptotic behaviour of a general solution of \eqref{1.1} expressed by the
Weierstrass $\wp$-function. An approach to such an expression was made by
Joshi and Kruskal \cite{J-K} by the method of multiple-scale expansions, 
the labelling of the solution by $(s_j)$ not being considered. Based on the
isomonodromy deformation of \eqref{1.2}, using WKB analysis, Kitaev \cite{K1},
\cite{K-3}, and Kapaev and Kitaev \cite{Ka-Ki} presented the elliptic 
asymptotic representation of a general solution of \eqref{1.1} 
as $x\to \infty$ through a cheese-like strip along any direction in the sector
$3\pi/5<\phi -2\pi k/5 <\pi,$ $k\in \mathbb{Z}.$ 
For $k=-2$ the asymptotic result is described as follows
\cite[p.~593, Section 8]{K1}, \cite[Theorem 2]{K-3}, \cite[Theorem 2]{Ka-Ki},  
in which $\omega_{\mathbf{a}},$ $\omega_{\mathbf{b}}$ and $A_{\phi}$ are,
respectively, the periods of an elliptic curve and a unique solution of the
Boutroux equations explained in {\bf (1)} at the end of this section.
\begin{thma}\label{thmA}
Let $y(x)$ be a solution of \eqref{1.1} corresponding to the monodromy data
$(s_j)_{1\le j\le 5}$ such that $s_1s_4\not=0,$ and let $\phi$ be such that 
$|\phi|<\pi/5.$
Then for any real numbers $c>0$ and $\varepsilon >0$ there exists a $\delta >0$
such that 
\begin{equation}\label{1.3}
y(x)=(e^{-i\phi}x)^{1/2}\left(\wp(e^{i\phi}t-t_0; g_2(\phi),g_3(\phi))
 +O(t^{-\delta})\right), \quad t=\tfrac 45 (e^{-i\phi}x)^{5/4}, 
\end{equation}
as $x\to \infty$ through $\mathcal{D}_{(s_j)}(\phi,c,\varepsilon),$ where
\begin{equation}\label{1.4}
t_0=\frac 1{2\pi i} \Bigl(\omega_{\mathbf{a}} \ln(is_1) +\omega_{\mathbf{b}}
\ln \frac{s_4}{s_1} \Bigr),
\end{equation}
$\wp(u;g_2,g_3)$ is the $\wp$-function such that
$(\wp_{u})^2=4\wp^3-g_2 \wp -g_3$ with $g_2(\phi)= -2e^{i\phi},$ $g_3(\phi)
=-A_{\phi}$, and $\mathcal{D}_{(s_j)}(\phi,c,\varepsilon)$ the cheese-like
strip defined by
$$
\bigl\{x=e^{i\phi}(\tfrac 54 t)^{4/5} \in \mathbb{C}\, : \,\,\, \re t>0, \,\,\,
|\im t|<c, \,\,\, |te^{i\phi} -t_0 +m\omega_{\mathbf{a}} +n\omega_{\mathbf{b}}|
>\varepsilon; \,\,\, m, n\in \mathbb{Z}\bigr\}.
$$
\end{thma}
\begin{rem}\label{rem1.1}
(1) The expression of   
$x\in\mathcal{D}_{(s_j)}(\phi,c,\varepsilon)$ (and that in 
Definition \ref{def2.1}) is in accordance with \cite[Theorem 2,~(1.38)]{K-3}, 
that is, each $x= e^{i\phi}(\tfrac 54 t)^{4/5} 
\in\mathcal{D}_{(s_j)}(\phi,c,\varepsilon)$ 
fulfils $|x|=(\tfrac 54|t|)^{4/5}$ and $\arg x=\phi +\tfrac 45 \arg t$, 
where $\arg t \ll |t|^{-1}$ as $t \to \infty$, since $|\im t|<c.$ 
The $\mathcal{D}_{(s_j)}(\phi,c,\varepsilon)$ contains the centre line
$\arg x=\phi.$ 
\par
(2) Let $\mathcal{S}(c):=\{ t\,|\, \re t>0, |\im t|<c\}$ be the strip
corresponding to $\mathcal{D}_{(s_j)}(\phi, c,\varepsilon)$. If 
$c>2(|\omega_{\mathbf{a}}|+|\omega_{\mathbf{b}}|)$, then there exists
a string of period parallelograms $\{L_N\}_{N=1}^{\infty}$ such that
$\bigcup_{N=1}^{\infty}L_N\subset \mathcal{S}(c)$ and that, for
every $N$, $(L_N)^{\mathrm{cl}} \cap (L_{N+1})^{\mathrm{cl}}$ is one of 
$[v^N_0,v^N_1]$, $[v^N_1, v^N_2]$, $[v^N_0, v^N_3],$ $[v^N_3, v^N_2],$ 
where $v^N_0,$ $v^N_1=v^N_0+\omega_{\mathbf{a}}$, 
$v^N_2=v^N_0+\omega_{\mathbf{a}}+\omega_{\mathbf{b}}$, 
$v^N_3=v^N_0+\omega_{\mathbf{b}}$ are the vertices of $L_N$, 
and $(L_N)^{\mathrm{cl}}$ denotes the closure of $L_N$. 
\par 
(3) Practically we may suppose $\varepsilon$ to be so small that, 
for all $(m,n) \in \mathbb{Z}^2$, the circles
$|te^{i\phi}-t_0+m\omega_{\mathbf{a}}+n\omega_{\mathbf{b}}|=\varepsilon$ are
disjoint.
\end{rem}
The leading term of \eqref{1.3} depends on the integration constant $t_0
=t_0((s_j)_{1\le j \le 5})$ only and the other one, which may be called 
the second integration constant, is hidden in the error term 
$O(t^{-\delta})$. Thus, in treating $y(x)$ as a general solution, it is 
preferable to know the explicit error term. 
This high-order part is also related, say, to
the $\tau$-function \cite[p.~121]{K-3} or to degeneration to trigonometric
asymptotics \cite[Section 4]{K-3}. For the $\tau$-function
associated with \eqref{1.1}, by the method of topological recursion, Iwaki 
\cite{Iwaki} obtained a conjectural full-order formal expansion yielding the
elliptic expression (see also \cite{IM}). 
\par
The justification procedure for the asymptotics of $y(x)$
\cite[pp.~105--106, pp.~120--121]{K-3}, \cite{K} is based on the leading terms 
of \eqref{3.1} and of the correction function $B(\phi,t)$, the pair of which
should depend on both integration constants labelling the solution $y(x)$. 
The second integration constant
is also related to $B(\phi,t)$ given in the following \cite[Theorem 4]{K-3}, 
where $J_{\mathbf{a}}$ and $\vartheta(z,\tau)$ with 
$\vartheta'(z,\tau)=(d/dz)\vartheta(z,\tau)$ are as in 
{\bf (1)} and {\bf (2)}. 
\begin{thmb}\label{thmB}
Let $B(\phi,t)= t((e^{-i\phi}x)^{-3/2}2\mathrm{H}(x,y(x),y'(x)) -A_{\phi})$,
where $\mathrm{H}(x,q,p)=p^2/2-2q^3-xq$ is the Hamiltonian of \eqref{1.1}. 
Then as $t \to \infty$ the function $B(\phi,t)$ is bounded uniformly
on the set $\mathcal{D}_{(s_j)}(\phi,c,\varepsilon)$ and is given by
\begin{equation}\label{1.5}
B(\phi,t)=\frac{8}{5\omega_{\mathbf{a}}} \Bigl(\ln \frac{s_1}{s_4} -\frac 54
tJ_{\mathbf{a}} +\pi i +\frac{\vartheta'}{\vartheta}\Bigl(\frac 1{\omega
_{\mathbf{a}}} (te^{i\phi} -t_0)+\nu ,\tau \Bigr) \Bigr)+O(t^{-\delta}),
\end{equation}
$\omega_{\mathbf{a}}$, $t_0$, $c$, $\varepsilon$ and $\delta$ being as 
in Theorem A.
\end{thmb}
\par
It is natural to seek explicitly the error term of \eqref{1.3}.
In this paper we present an explicit asymptotic representation of the 
error term, and this formula 
leads to the error bound $O(t^{-1})=O(x^{-5/4})$ in \eqref{1.3}.
The main results are
stated in Theorems \ref{thm2.1} and \ref{thm2.2}. Corollary
\ref{cor2.3} describes the dependence on the other integration
constant $\beta_0$ in this explicit error term. 
In Section \ref{sc4} these results are derived 
from the $t^{-\delta}$-asymptotics of $y(x)$ in Theorem A and 
of $B(\phi,t)$ in Theorem B combined with a system of integral equations 
equivalent to \eqref{1.1}, which is constructed in Section \ref{sc3}. 
Necessary lemmas in our argument are shown in Section \ref{sc5}. 
The final section is devoted to discussions on 
our method in deriving the explicit error with the bound $O(t^{-1})$, 
which is quite 
different from that in \cite{I-K}, \cite{J-K} and \cite{Novo}.
\par 
Throughout this paper we use the symbols summed up below 
(cf. \cite{K1}, \cite{K-3}).
\par
{\bf (1)} For each $\phi\in \mathbb{R}$, the condition
$$
\re \int_{\mathbf{c}} w d\lambda =0,
$$
where $\mathbf{c}$ is any cycle on the elliptic curve 
$w^2=\lambda^3+\tfrac 12 e^{i\phi}\lambda +\tfrac 14A,$ uniquely
determines $A=A_{\phi} \in \mathbb{C}$ \cite[Lemma 3]{K-3}, \cite[Section 7]{K1}
(The conditions for basic cycles $\mathbf{c}=\mathbf{c}_1,$ $\mathbf{c}_2$
are the Boutroux equations). Then $A_{\phi}$ is continuous in $\phi$ 
and smooth for $\phi \not=\pi +2\pi k/5,$ $(k\in \mathbb{Z})$, and fulfils
the relations $A_{-\phi}=\overline{A_{\phi}},$ $ A_{\phi +2\pi}=A_{\phi},$ 
$A_{\phi -2\pi k/5}= e^{2\pi i k/5}A_{\phi},$ and 
$A_{\pi}=-(\tfrac 23)^{3/2},$ $ 0.336 <A_0<0.380$ 
\cite[Lemma 4]{K-3}, \cite[Section 7]{K1}. Suppose that $|\phi|<\pi/5$, and
set $w=w(A_{\phi},\lambda)=\sqrt{\lambda^3+\tfrac 12 e^{i\phi}\lambda 
+\tfrac 14 A_{\phi}}=\sqrt{\lambda-\lambda_1}\,\sqrt{\lambda-\lambda_2}\,
\sqrt{\lambda-\lambda_3}.$ 
Let $\mathbf{a}$ and $\mathbf{b}$ be basic cycles as in Figure \ref{cycles}
on the elliptic curve $\Gamma=\Gamma_+ \cup \Gamma_-$ defined by $w(A_{\phi},
\lambda)$ such that the upper and lower sheets $\Gamma_+$ and $\Gamma_-$ are 
glued along the cuts
$[\lambda_1,\lambda_2]\cup[-\infty,\lambda_3]$; where the branch points 
$\lambda_j=\lambda_j(\phi)$ for $\phi$ around $\phi=0$ are specified 
in such a way that
$\re \lambda_1(0)>0,$ $\im \lambda_1(0)>0,$ 
$\lambda_2(0)=\overline{\lambda_1(0)},$ $\lambda_3(0)<0,$ and the branches of
$\sqrt{\lambda-\lambda_j}$ are determined by $\sqrt{\lambda-\lambda_j}
\to +\infty$ as $\lambda\to +\infty$ along the positive real axis on $\Gamma_+.$
{\small
\begin{figure}[htb]
\begin{center}
\unitlength=0.8mm
\begin{picture}(80,55)(-40,-25)
\put(17.5,13){\makebox{$\lambda_1$}}
\put(8,-15){\makebox{$\lambda_2$}}
\put(-10.5,-6.5){\makebox{$\lambda_3$}}
\thinlines
\put(15.5, -15){\line(0,1){30}}
\put(14.5,-15){\line(0,1){30}}
\put(-10,0.5){\line(-1,0){30}}
\put(-10,-0.5){\line(-1,0){30}}
 \qbezier(28,15) (30.1,10) (30.3,5)
 \qbezier(-18,13.6) (-13,15.2) (-7,12.8)
 \put(28,15){\vector(-2,3){0}}
 \put(-7,12.8){\vector(3,-2){0}}
 \put(30,1){\makebox{$\mathbf{a}$}}
 \put(-23,12){\makebox{$\mathbf{b}$}}
 \put(35,-25){\makebox{$\Gamma_{+}$}}
\thicklines
\put(15,15){\circle*{1.5}}
\put(15,-15){\circle*{1.5}}
\put(-10,0){\circle*{1.5}}
 \qbezier(5,15) (15,32) (25,15)
 \qbezier(5,15) (-1,0) (5,-15)
 \qbezier(25,15) (31,0) (25,-15)
 \qbezier(5,-15) (15,-32) (25,-15)

 \qbezier(-20.0,0.2) (-24,18) (0.7,7)
 \qbezier(14.714,-2.86) (12.0,1.0) (4,5.3)
 \qbezier[12](23,-15.6) (22.0,-9.8) (16.5,-4.0)
 \qbezier[12](23.5,-17.5) (25.5,-27) (15,-26)
 \qbezier[35](15,-26.0)(-10,-24)(-19.5,-1.5)

\end{picture}
\end{center}
\caption{Cycles $\mathbf{a}$ and $\mathbf{b}$}
\label{cycles}
\end{figure}
}
For the cycles $\mathbf{a}$ and $\mathbf{b}$ on $\Gamma$ write
\begin{align*}
& \omega_{\mathbf{a}}=\omega_{\mathbf{a}}(\phi)=\frac 12\int_{\mathbf{a}}
\frac{d\lambda}{w}, \quad
\omega_{\mathbf{b}}=\omega_{\mathbf{b}}(\phi)=\frac 12\int_{\mathbf{b}}
\frac{d\lambda}{w}, \quad
\\
& J_{\mathbf{a}}=J_{\mathbf{a}}(\phi)=2 \int_{\mathbf{a}} w d\lambda,
 \quad
 J_{\mathbf{b}}=J_{\mathbf{b}}(\phi)=2 \int_{\mathbf{b}} w d\lambda,
\end{align*}
where $\omega_{\mathbf{a}},$ $\omega_{\mathbf{b}}$ fulfil $\im 
\omega_{\mathbf{b}}/{\omega_{\mathbf{a}}} >0.$ For $\phi \not= \pi+2\pi k/5$
the elliptic curve $\Gamma$ with the cycles $\mathbf{a},$ $\mathbf{b}$ and
the integrals $\omega_{\mathbf{a},\,\mathbf{b}}$ and $J_{\mathbf{a},\,
\mathbf{b}}$ may be extended by the use of the substitution $(\phi, \lambda, w)
\mapsto (\phi+2\pi/5, e^{-4\pi i/5}\lambda, e^{-6\pi i/5}w)$ \cite[pp.~95--96]
{K-3}. 
\par
{\bf (2)} For $\tau=\omega_{\mathbf{b}}/\omega_{\mathbf{a}}$ such that 
$\im \tau>0$, 
$$
\vartheta(z,\tau)=\sum_{n\in \mathbb{Z}}e^{\pi i\tau n^2+2\pi izn}
$$ 
is the theta-function \cite{H}, \cite{WW} with 
$\nu=(1+\tau)/2,$ which satisfies $\vartheta(z\pm 1,
\tau)=\vartheta(z,\tau),$ $\vartheta(z\pm \tau,\tau)=e^{-\pi i(\tau\pm 2z)}
\vartheta(z,\tau).$ 
\par
{\bf (3)} We write $f \ll g$ or $g \gg f$ if $f=O(g)$, and $f \asymp g$
if $f\ll g$ and $g\ll f.$
\section{Main results}\label{sc2}
To state our results let us define some strips similar to $\mathcal{D}
_{(s_j)}(\phi,c,\varepsilon)$.
Recall the phase shift $t_0$ given by \eqref{1.4}, and the zeros $\lambda_j$
$(1\le j\le 3)$ of $w(A_{\phi},\lambda)^2$ specified as in {\bf (1)}.
\begin{defi}\label{def2.1}
For given $t_{\infty}>0,$ $c>0$ and small $\varepsilon>0$, set
\begin{align*}
 \mathcal{D}(\phi,t_{\infty},c ,\varepsilon )=&
  \{ x= e^{i\phi} (\tfrac 54 t)^{4/5}
\in \mathbb{C}\, :
\\
& \, \re t>t_{\infty},\,\,  |\im t|<c,\,\, |e^{i\phi}t-t_0-m\omega
_{\mathbf{a}}-n \omega_{\mathbf{b}}| >\varepsilon, \, m,n \in \mathbb{Z}\}, 
\\
 \check{\mathcal{D}}(\phi,t_{\infty},c ,\varepsilon )=&
  \mathcal{D}(\phi, t_{\infty},c,\varepsilon )
\setminus \bigcup_{1\le j\le 3;\, m,n\in \mathbb{Z}} 
\mathcal{V}_{j,m,n}(\varepsilon),
\\
\mathcal{V}_{j,m,n}(\varepsilon)
=& \{ x= e^{i\phi} (\tfrac 54 t)^{4/5}\in \mathbb{C}\,:\, |e^{i\phi}(t-t_j)
-t_0-m\omega_{\mathbf{a}} -n\omega_{\mathbf{b}} | \le \varepsilon   \},
\end{align*}
where $t_j$ $(1\le j\le 3)$ are such that $\wp(e^{i\phi}t_j -t_0; g_2(\phi),
g_3(\phi))=\lambda_j$ (for $x=e^{i\phi}(\tfrac 54t)^{4/5}$ 
in these domains, cf.~Remark \ref{rem1.1}). 
\end{defi}
Let $t_j^{m,n}\in \mathcal{S}(t_{\infty},c):=\{t\,|\, \re t >t_{\infty},
|\im t|<c\}$ $(j=1,2,3,\infty)$ be such that 
$t_j^{m,n}=t_j+e^{-i\phi}(t_0+m\omega_{\mathbf{a}}+n\omega_{\mathbf{b}})$ 
if $1\le j\le 3$, and that $t_{\infty}^{m,n}=e^{-i\phi}(t_0+m\omega_{\mathbf
{a}}+n\omega_{\mathbf{b}})$ if $j=\infty$.
We may suppose $\varepsilon$ so small that all the circles $|t-t_j^{m,n}|=
\varepsilon$ $(j=1,2,3,\infty; m,n\in \mathbf{Z})$ are disjoint.
Let $l(t^{m,n}_j)\in \mathcal{S}(t_{\infty},c)$ $(j=1,2,3,\infty)$ 
be the lines defined by $t=\re t^{m,n}_j +i\eta$, $\eta \ge \im 
t^{m,n}_j$ if $\im t^{m,n}_j \ge 0$ 
(respectively, $\eta \le \im t^{m,n}_j$ if $\im t^{m,n}_j <0$), and let
$l^x(t^{m,n}_j)$ be the images of $l(t^{m,n}_j)$ under the mapping 
$x=x(t)=e^{i\phi}(\tfrac 54 t)^{4/5}$.
\begin{defi}\label{def2.2}
For $t_{\infty},$ $c$, $\varepsilon$ as in Definition \ref{def2.1},
$\check{\mathcal{D}}_{\mathrm{cut}}(\phi,t_{\infty},c,\varepsilon)$ denotes
$\check{\mathcal{D}}(\phi,t_{\infty},c,\varepsilon)$ having cuts along  
$l^x(t^{m,n}_j)$ for all $j=1,2,3, \infty$ and $m,n \in \mathbb{Z},$ where
some local segments contained in
$l^x(t^{m,n}_j)$ are replaced with suitable arcs, if necessary, 
not to touch the image of another small circle $|t-t^{m',n'}_{j'}|=\varepsilon$
with $(j',m',n')\not=(j,m,n)$ under the mapping $x(t).$ 
\end{defi}
{\small
\begin{figure}[htb]
\begin{center}
\unitlength=0.8mm
\begin{picture}(80,75)(-40,-45)

  \qbezier (-27,11) (0,20) (27,29)
  \qbezier (-27,-29) (0,-20) (27,-11)

\put(20,10){\circle*{0.7}} \put(20,10){\circle{3}}
\put(20,-10){\circle*{0.7}} \put(20,-10){\circle{3}}
\put(-20,10){\circle*{0.7}} \put(-20,10){\circle{3}}
\put(-20,-10){\circle*{0.7}} \put(-20,-10){\circle{3}}
\put(0,10){\circle*{0.7}} \put(0,10){\circle{3}}
\put(0,-10){\circle*{0.7}} \put(0,-10){\circle{3}}

\put(0,0){\circle*{0.7}} \put(0,0){\circle{3}}
\put(20,0){\circle*{0.7}} \put(20,0){\circle{3}}
\put(20,20){\circle*{0.7}} \put(20,20){\circle{3}}
\put(-20,0){\circle*{0.7}} \put(-20,0){\circle{3}}
\put(-20,-20){\circle*{0.7}} \put(-20,-20){\circle{3}}

\put(10,0){\circle*{0.7}} \put(10,0){\circle{3}}
\put(10,10){\circle*{0.7}} \put(10,10){\circle{3}}
\put(10,20){\circle*{0.7}} \put(10,20){\circle{3}}
\put(10,-10){\circle*{0.7}} \put(10,-10){\circle{3}}

\put(-10,0){\circle*{0.7}} \put(-10,0){\circle{3}}
\put(-10,10){\circle*{0.7}} \put(-10,10){\circle{3}}
\put(-10,-20){\circle*{0.7}} \put(-10,-20){\circle{3}}
\put(-10,-10){\circle*{0.7}} \put(-10,-10){\circle{3}}
\put(-10,-20){\circle*{0.7}} \put(-10,-20){\circle{3}}

 \put(27,-1){\makebox{$x(t^{m,n}_j)$}}

\put(-20,-43){\makebox{(a) \quad$\check{\mathcal{D}}(\phi,t_{\infty},
c,\varepsilon)$}}
\end{picture}
\quad\quad
\begin{picture}(80,75)(-40,-45)

  \qbezier (-27,11) (0,20) (27,29)
  \qbezier (-27,-29) (0,-20) (27,-11)

\put(20,20){\circle*{0.7}} \put(20,20){\circle{3}}
\put(20,10){\circle*{0.7}} \put(20,10){\circle{3}}
\put(20,-10){\circle*{0.7}} \put(20,-10){\circle{3}}
\put(-20,10){\circle*{0.7}} \put(-20,10){\circle{3}}
\put(-20,-10){\circle*{0.7}} \put(-20,-10){\circle{3}}
\put(-20,-20){\circle*{0.7}} \put(-20,-20){\circle{3}}
\put(0,10){\circle*{0.7}} \put(0,10){\circle{3}}
\put(0,-10){\circle*{0.7}} \put(0,-10){\circle{3}}

\put(0,0){\circle*{0.7}} \put(0,0){\circle{3}}
\put(20,0){\circle*{0.7}} \put(20,0){\circle{3}}
\put(-20,0){\circle*{0.7}} \put(-20,0){\circle{3}}

\put(10,0){\circle*{0.7}} \put(10,0){\circle{3}}
\put(10,10){\circle*{0.7}} \put(10,10){\circle{3}}
\put(10,20){\circle*{0.7}} \put(10,20){\circle{3}}
\put(10,-10){\circle*{0.7}} \put(10,-10){\circle{3}}

\qbezier (-20.47,11.42) (-20.97,12.92)  (-21.03,13.10)
\qbezier (20.47,-11.42) (20.97,-12.92)  (21.03,-13.10)

\qbezier (9.53,21.42) (9.03,22.92)  (8.97,23.10)
\qbezier (-9.53,-21.42) (-9.03,-22.92)  (-8.97,-23.10)

\put(9.53,11.42){\line(-1,3){3.5}} 
\put(-9.53,-11.42){\line(1,-3){3.5}} 
\put(10.47,-1.42){\line(1,-3){4.5}} 
\put(-10.47,1.42){\line(-1,3){4.5}} 
\put(10.47,-11.42){\line(1,-3){1.5}} 
\put(-10.47,11.42){\line(-1,3){1.5}}

\put(-0.53,11.42){\line(-1,3){2.5}} 
\put(-0.53,1.42){\line(-1,3){5.5}} 
\put(0.47,-11.42){\line(1,-3){2.5}} 

\put(-20.47,1.42){\line(-1,3){3.5}} 
\put(20.47,-1.42){\line(1,-3){3.5}} 
\put(-19.47,-11.42){\line(1,-3){4.5}} 
\put(-19.47,-21.42){\line(1,-3){1.5}} 
\put(19.47,11.42){\line(-1,3){4.5}} 
\put(19.47,21.42){\line(-1,3){1.5}} 

\put(-10,0){\circle*{0.7}} \put(-10,0){\circle{3}}
\put(-10,10){\circle*{0.7}} \put(-10,10){\circle{3}}
\put(-10,-20){\circle*{0.7}} \put(-10,-20){\circle{3}}
\put(-10,-10){\circle*{0.7}} \put(-10,-10){\circle{3}}

 \put(28,-5){\makebox{$l^x(t^{m,n}_j)$}}

\put(-20,-43){\makebox{(b) \quad$\check{\mathcal{D}}_{\mathrm{cut}}
(\phi,t_{\infty},c,\varepsilon)$}}
\end{picture}
\end{center}
\caption{$\check{\mathcal{D}}(\phi,t_{\infty},c,\varepsilon)$ and 
$\check{\mathcal{D}}_{\mathrm{cut}}(\phi,t_{\infty},c,\varepsilon)$} 
\label{strips}
\end{figure}
}

These domains are strips lying in the
$x$-plane. For a function $f(x)$ we may deal with $\hat{f}(t)=f(x(t))$ with 
$t=\tfrac 45 (e^{-i\phi}x)^{5/4}$ or $\hat{f}(z)=f(x(z))$ with $z=e^{i\phi}t
=\tfrac 45 e^{i\phi}(e^{-i\phi}x)^{5/4}$ as $x\to \infty$, say, through
$\check{\mathcal{D}}_{\mathrm{cut}}(\phi,t_{\infty}, c,\varepsilon)$, and then
we write, in short, $f(t)$ or $f(z)$ in 
$\check{\mathcal{D}}_{\mathrm{cut}}(\phi,t_{\infty}, c,\varepsilon)$. 

Suppose that $|\phi|<\pi/5.$ Let $y(x)$ be a solution of \eqref{1.1}
corresponding to $(s_j)_{1\le j\le 5}$ such that $s_1s_4 \not=0.$
Then we have the following.
\begin{thm}\label{thm2.1}
Let $c>0$ and $\varepsilon>0$ be given positive numbers as in Theorem A. Then
$$
y(x)=(e^{-i\phi}x)^{1/2}
 \left(\wp(e^{i\phi}t-t_0; g_2(\phi),g_3(\phi))+O(t^{-1})\right)
$$
as $x=e^{i\phi}(\tfrac 54 t)^{4/5} \to \infty$ through $\mathcal{D}(\phi,
t_{\infty}, c,\varepsilon)$, where $t_{\infty}$ is sufficiently large. 
\end{thm}
To deal with the error term explicitly, let us write \eqref{1.3} in the form
$$
y(x)=(e^{-i\phi}x)^{1/2} \wp(e^{i\phi}t-t_0 +h(e^{i\phi}t);g_2(\phi),g_3(\phi)).
$$
Set $z=e^{i\phi}t$ and write
\begin{align*}
& \mathfrak{p}=\mathfrak{p}(z)=\wp(z-t_0; g_2(\phi),g_3(\phi)),
\\
&\beta=\beta(z)=\frac{8e^{i\phi}}{5\omega_{\mathbf{a}}}\Bigl(\beta_0 -\frac 54
e^{-i\phi}J_{\mathbf{a}} (z-t_0) +\frac{\vartheta'}{\vartheta}\Bigl(\frac{z-t_0}
{\omega_{\mathbf{a}}} +\nu,\tau \Bigr) \Bigr),
\\
& \beta_0=\ln \frac{s_1}{s_4}-\frac 54 e^{-i\phi}J_{\mathbf{a}} t_0 +\pi i.
\end{align*}
The quantity  
$\beta(z)=e^{i\phi}B_{\mathrm{as}}(\phi,t)$ is the leading term of the
correction function $e^{i\phi}B(\phi,t)$ in Theorem B,
and $\beta(z)$ and $\mathfrak{p}(z)$ 
are bounded uniformly in $\mathcal{D}(\phi,t_{\infty},c,\varepsilon)$.
Then we have the following explicit expression of $h(z).$
\begin{thm}\label{thm2.2}
Let $P(\lambda)=4w(A_{\phi},\lambda)^2=4\lambda^3 +2e^{i\phi}\lambda+A_{\phi},$
and let $c,$ $\varepsilon$, $\delta$ be as in Theorem A.
Then as $x=e^{i\phi}(\tfrac 54 e^{-i\phi}z)^{4/5}
\to \infty$ through $\check{\mathcal{D}}_{\mathrm{cut}}(\phi,t_{\infty},
c,\varepsilon)$,
\begin{align*}
h(z)= &  \int^z_{\infty} F(\mathfrak{p},\beta)\frac{d\zeta}{\zeta}
-\int^z_{\infty} (F(\mathfrak{p},\beta)^2+G(\mathfrak{p},\beta)) \frac{d\zeta}{\zeta^2}
\\
& -\frac 1{10} \int^z_{\infty} K(\mathfrak{p},\beta) \mathcal{I}(\zeta) 
\frac{d\zeta}{\zeta} +O(z^{-1-\delta}),
\end{align*}
in which
\begin{align*}
F(\mathfrak{p},\beta) &=\frac{\beta}{2P(\mathfrak{p})}-\frac{2\mathfrak{p}}{5\sqrt{P(\mathfrak{p})}},
\quad
G(\mathfrak{p},\beta)=\frac{\beta^2}{8P(\mathfrak{p})^2},
\\
K(\mathfrak{p},\beta) &= 2(4e^{i\phi}\mathfrak{p} +3A_{\phi})F(\mathfrak{p},\beta)-\beta, \quad
\mathcal{I}(z)=\int^z_{\infty} \frac 1{P(\mathfrak{p})}\frac{d\zeta}{\zeta}
\end{align*}
with $\mathfrak{p}=\mathfrak{p}(\zeta),$ $\beta=\beta(\zeta)$, and in 
$\check{\mathcal{D}}_{\mathrm{cut}}(\phi,t_{\infty},c,\varepsilon)$ 
the integrals converge and are evaluated as 
\begin{align*}
&\mathcal{I}(z) \ll z^{-1}, \quad \int^z_{\infty} F(\mathfrak{p},\beta)\frac{d\zeta}
{\zeta} \ll z^{-1},
\\
&\int^z_{\infty}(F(\mathfrak{p},\beta)^2+G(\mathfrak{p},\beta))\frac{d\zeta}{\zeta^2}\ll z^{-1},
\quad  \int^z_{\infty} K(\mathfrak{p},\beta)\mathcal{I}(\zeta) \frac{d\zeta}{\zeta}
\ll z^{-1}.
\end{align*} 
\end{thm}
\begin{rem}\label{rem2.1}
For the sake of convenience the integrals are considered in 
$\check{\mathcal{D}}_{\mathrm{cut}}(\phi,t_{\infty},c,\varepsilon)$,
which is simply connected, to avoid the 
possible multi-valuedness around poles of the integrands. Then 
$h(z)$ is considered substantially along the ray $\arg x=\phi$ and expressed 
by the $\wp$-function with $(g_2(\phi),g_3(\phi))=(-2e^{i\phi},-A_{\phi})$. 
\end{rem}
The integration constant $\beta_0$ appears in $h(z)$ as described in the 
following.
\begin{cor}\label{cor2.3}
In $\check{\mathcal{D}}_{\mathrm{cut}}(\phi,t_{\infty},c,\varepsilon)$  
we have
$zh(z)=h_0 \beta_0^2 + h_1(z) \beta_0 +h_2(z)+O(z^{-\delta})$
with $h_1(z)=O(1),$ $h_2(z)=O(1)$ and 
$$
h_0= -\frac {24}{5\omega_{\mathbf{a}}^2} e^{2i\phi}(8e^{3i\phi}+27A_{\phi}
^2)^{-1}. 
$$
\end{cor} 
\begin{rem}\label{rem2.3}
Along any direction in the sector $3\pi/5<\phi-2\pi k/5<\pi,$ $k\in\mathbb{Z}$,
the formulas of the main results are written in terms of $t_0$ and $\beta_0$ 
with $(s_{2-2k}, s_{5-2k}/s_{2-2k})$ in place of $(s_1, s_4/s_1)$ 
(cf.~\cite[Theorems 2 and 4]{K-3}). This discontinuity of integration 
constants may also be considered as a nonlinear Stokes phenomenon 
\cite[p.~379]{FIKN}. 
\end{rem}
\begin{rem}\label{rem2.4}
For a given simply connected unbounded subdomain $\mathcal{D}_0 \subset
\check{\mathcal{D}}(\phi,t_{\infty},c,\varepsilon)$, Theorem \ref{thm2.2} and 
Corollary \ref{cor2.3} are also valid as $x\to \infty$ through $\mathcal{D}_0.$
\end{rem}
\section{System of equations and integral representations}\label{sc3}
Let $c$, $\varepsilon$ and $\delta$ be as in Theorem A. The number $t_{\infty}$
is retaken if necessary, being denoted by the same letter $t_{\infty}$ in
each appearance in our argument.
In the proofs of our theorems we may suppose that $0<\delta <1.$
\par
By the change of variables $t=\tfrac 45 (e^{-i\phi}x)^{5/4},$
$y=(e^{-i\phi}x)^{1/2}v$, equation \eqref{1.1} or $(d/dx)((y')^2-4y^3-2xy)=-2y$ 
is taken to
$$
t\frac d{dt}a_{\phi} +\frac 65 a_{\phi}= -\frac 85 e^{i\phi}v
$$
with
$$
a_{\phi}=e^{-2i\phi}\bigl(v_t+(2/5)t^{-1}v\bigr)^2 -4v^3-2e^{i\phi}v,
$$
which is the Lagrangian or the Hamiltonian 
$2(e^{-i\phi}x)^{-3/2}\mathrm{H}(x,y(x),y'(x))$ appearing in Theorem B,  
\cite[Theorem 4]{K-3}. 
Setting $a_{\phi}=A_{\phi}+t^{-1}B(\phi,t)$ with $e^{i\phi}t=z,$
$e^{i\phi}B(\phi,t)=b,$ we have the system of equations 
\begin{equation}\label{3.1}
\begin{split}
&\bigl(v_z +(2/5) z^{-1}v\bigr)^2=4v^3+2 e^{i\phi}v+ A_{\phi}+z^{-1}b,
\\
&b_z= -\frac 85e^{i\phi}v -\frac 65 A_{\phi} -\frac 15 z^{-1}b.
\end{split}
\end{equation}
Note that \eqref{1.5} is written as $B(\phi,t)=B_{\mathrm{as}}(\phi,t)
 +O(t^{-\delta})$. The system
\begin{equation}\label{3.2}
\begin{split}
&\mathfrak{p}_z^2 =4\mathfrak{p}^3+2e^{i\phi}\mathfrak{p} +A_{\phi},
\\
&\beta_z= -\frac 85e^{i\phi}\mathfrak{p} -\frac 65 A_{\phi}
\end{split}
\end{equation}
admits a solution $(\mathfrak{p},\beta)=(\mathfrak{p}(z),\beta(z))
=(\wp(z-t_0;g_2(\phi), g_3(\phi)), e^{i\phi}B_{\mathrm{as}}(\phi,e^{-i\phi}z))$,
where validity of the second equation is due to
\cite[Proposition 6, (3.13)]{K-3}. System \eqref{3.2} is, at least formally, 
an approximation to \eqref{3.1}. 
\begin{prop}\label{prop3.1}
For $y(x)$ given by \eqref{1.3} set $\mathfrak{p}(z+h(z))=(e^{-i\phi}x)
^{-1/2}y(x)$ with $e^{i\phi}t=z$ in $\check{\mathcal{D}}_{\mathrm{cut}}
(\phi,t_{\infty},c,\varepsilon)$. Then $(v,b)=(\mathfrak{p}(z+h(z)), 
e^{i\phi}B(\phi,e^{-i\phi}z))$ solves \eqref{3.1}.
\end{prop}
By \eqref{1.3} and \cite[Theorem 4]{K-3} with
$b(z)-\beta(z)=e^{i\phi}(B(\phi,e^{-i\phi}z)-B_{\mathrm{as}}(\phi,
e^{-i\phi}z))$, we have the following.
\begin{prop}\label{prop3.2}
In $\check{\mathcal{D}}_{\mathrm{cut}}(\phi,t_{\infty},c,\varepsilon)$,
$h(z) \ll z^{-\delta},$ and $b(z)$ is bounded and satisfies
$b(z)-\beta(z) \ll z^{-\delta}.$
\end{prop}
Recall that $P(\lambda)=4\lambda^3+2e^{i\phi}\lambda +A_{\phi}.$
Let us insert $v={\mathfrak{p}_{(h)}}(z):=\mathfrak{p}(z+h(z))$ into the first equation of 
\eqref{3.1}. Observing that $v_z=(1+h')\mathfrak{p}'(z+h)=(1+h')
\sqrt{P({\mathfrak{p}_{(h)}})},$ and that $(1+h')\sqrt{P({\mathfrak{p}_{(h)}})}+(2/5)z^{-1}
{\mathfrak{p}_{(h)}}=\sqrt{P({\mathfrak{p}_{(h)}})} (1+z^{-1}b P({\mathfrak{p}_{(h)}})^{-1})^{1/2},$
we have, in $\check{\mathcal{D}}(\phi,t_{\infty},c,\varepsilon)$,
$$
h'=\Bigl( \frac{b}{2P({\mathfrak{p}_{(h)}})}
-\frac{2{\mathfrak{p}_{(h)}}}{5\sqrt{P({\mathfrak{p}_{(h)}})}}
\Bigr)z^{-1} - \frac{b^2}{8P({\mathfrak{p}_{(h)}})^2} z^{-2} +O(z^{-3}).
$$
Furthermore, by $\mathfrak{p}_{(h)}(z)=\mathfrak{p}(z)+h \mathfrak{p}'(z)+O(h^2),$
\begin{equation}\label{3.3}
h'=F(\mathfrak{p},b)z^{-1}-G(\mathfrak{p},b)z^{-2} +F_v(\mathfrak{p},b)\mathfrak{p}'hz^{-1} +O(z^{-1}(|z^{-1}|
+|h|)^2)
\end{equation}
with $F(v,b)=\tfrac 12 bP(v)^{-1}-\tfrac 25 vP(v)^{-1/2},$ $G(v,b)=\tfrac 18
b^2P(v)^{-2}$ as in Theorem \ref{thm2.2}.
Write $\chi:=b-\beta.$ Then the second equations of \eqref{3.1} and \eqref{3.2}
yield
\begin{align*}
\chi'=& -\frac 85 e^{i\phi}(\mathfrak{p}_{(h)}(z)-\mathfrak{p}(z)) -\frac{b(z)}5z^{-1}
\\
=& -\frac 85 e^{i\phi} \Bigl(\mathfrak{p}' h+\frac{\mathfrak{p}''}{2!}h^2 +\cdots + \frac
{\mathfrak{p}^{(m)}}{m!} h^m +E_m h^{m+1} \Bigr) -\frac{b}5z^{-1}
\end{align*}
for any positive integer $m$, where $E_m \ll 1$ in $\check{\mathcal{D}}(\phi,
t_{\infty},c,\varepsilon).$
Let $\{z_n\} \subset \check{\mathcal{D}}(\phi,t_{\infty},c,\varepsilon)$ 
be any sequence
such that $z_n \to\infty$. Then integration by parts leads to
\begin{align*}
\chi(z)-\chi(z_n) =& -\frac 85 e^{i\phi}\Bigl[\mathfrak{p} h +\frac{\mathfrak{p}'}2 h^2+
\cdots + \frac{\mathfrak{p}^{(m-1)}}{m!}h^m  \Bigr]^z_{z_n}
\\
 +\frac 85 e^{i\phi}  \int^z_{z_n} &
\Bigl(\Bigl(\mathfrak{p} +{\mathfrak{p}'} h+
\cdots + \frac{\mathfrak{p}^{(m-1)}}{(m-1)!}h^{m-1} \Bigr)h'-E_mh^{m+1} \Bigr)d\zeta
 -\frac 15 \int^z_{z_n}(\beta+\chi)\frac{d\zeta}{\zeta}.
\end{align*}
By \eqref{3.3}, $h'=(F(\mathfrak{p},\beta) +\tfrac 12 \chi P(\mathfrak{p})^{-1}+\tilde{R}
 )z^{-1}$ with $\tilde{R} \ll |z^{-1}|+|h|,$ and hence the sum of the
integrals on the right-hand side is
$$
 \frac 15 \int^z_{z_n}  (8e^{i\phi} F(\mathfrak{p},\beta)\mathfrak{p}-\beta)\frac{d\zeta}{\zeta}
+\frac 15 \int^z_{z_n}R(h,\chi) \frac{d\zeta}{\zeta} -\frac 85 e^{i\phi}
\int^z_{z_n}E_m h^{m+1} d\zeta,
$$
with
\begin{align}
\notag
R(h,\chi):=& 8e^{i\phi}\Bigl(\mathfrak{p}'+\cdots+\frac{\mathfrak{p}^{(m-1)}}{(m-1)!}h^{m-2}
\Bigr) \Bigl(F(\mathfrak{p},\beta)+\frac{\chi}{2P(\mathfrak{p})}+\tilde{R} \Bigr)h
\\
\label{3.4}
& +8e^{i\phi}\mathfrak{p}\Bigl(\frac {\chi}{2P(\mathfrak{p})}+\tilde{R} \Bigr)-\chi
\ll |h|+|\chi|+|\zeta^{-1}|.
\end{align}
Now suppose that, for a positive number $\mu$ satisfying $\delta\le \mu \le 1,$
\begin{equation}\label{3.5}
h(z) \ll z^{-\mu}
\end{equation}
in $\check{\mathcal{D}}_{\mathrm{cut}}(\phi,t_{\infty},c,\varepsilon).$ 
By Proposition \ref{prop3.2} this supposition is true for $\mu=\delta,$ and 
$\chi(z)\ll z^{-\delta}.$ Choose $m$ such that $\delta (m+1) \ge 3$. 
The passage to the limit $z_n\to \infty$ leads to
\begin{equation}\label{3.6}
\chi+\frac 85 e^{i\phi}\mathfrak{p} h =\frac 15 \int^z_{\infty} (8e^{i\phi}F(\mathfrak{p},\beta)
\mathfrak{p}-\beta)\frac{d\zeta}{\zeta}+\frac 15 \int^z_{\infty}R(h,\chi)\frac{d\zeta}
{\zeta} +O(z^{-2\mu}),
\end{equation}
in which the convergence of
$$
\int^z_{\infty} (8e^{i\phi}F(\mathfrak{p},\beta)\mathfrak{p}-\beta)\frac{d\zeta}{\zeta} \ll
z^{-\delta}
$$
is guaranteed by the absolute convergence of $\int^z_{\infty} R(h,\chi)
 \zeta^{-1} d\zeta \ll z^{-\delta}$ (cf. \eqref{3.4} and Proposition 
\ref{prop3.2}).
Under \eqref{3.5}, observing that, in \eqref{3.3},
$$
-G(\mathfrak{p},b)z^{-2} +F_v(\mathfrak{p},b)\mathfrak{p}'hz^{-1}=-G(\mathfrak{p},\beta)z^{-2}+F_v(\mathfrak{p},\beta)
\mathfrak{p}'hz^{-1}+O(z^{-1-\mu-\delta}),
$$
and that
\begin{align*}
\int^z_{z_n} F_v(\mathfrak{p},\beta)\mathfrak{p}' h\frac{d\zeta}{\zeta} =&
\int^z_{z_n} \Bigl(F(\mathfrak{p},\beta)'-\frac{\beta'}{2P(\mathfrak{p})} \Bigr)h\frac{d\zeta}
{\zeta} 
\\
=&\Bigl[ F(\mathfrak{p},\beta)h \zeta^{-1} \Bigr]^z_{z_n} -\int^z_{z_n} 
 \Bigl(F(\mathfrak{p},\beta)h'+\frac{\beta'h}{2P(\mathfrak{p})} \Bigr)\frac{d\zeta}{\zeta}
+\Bigl[ O(z^{-1-\mu})\Bigr]^z_{z_n}
\\
=& -\int^z_{z_n} 
 \Bigl(F(\mathfrak{p},\beta)^2\zeta^{-1}+\frac{\beta'h}{2P(\mathfrak{p})} \Bigr)\frac{d\zeta}
{\zeta}+ \Bigl[O(z^{-1-\delta})\Bigr]^z_{z_n},
\end{align*}
we may apply a similar argument to \eqref{3.3} with $F(\mathfrak{p},b)=F(
\mathfrak{p},\beta) +\tfrac 12 \chi P(\mathfrak{p})^{-1}$, and 
the convergence of 
$\int^z_{\infty} F(\mathfrak{p},\beta)\zeta^{-1}d\zeta$ follows. Thus we have
the following relations, in which the second equation follows from \eqref{3.6}
and \eqref{3.2}.
\begin{prop}\label{prop3.3}
Under the supposition \eqref{3.5}, in $\check{\mathcal{D}}_{\mathrm{cut}}
(\phi,t_{\infty},c,\varepsilon)$,
\begin{align*}
h=& \int^z_{\infty} F(\mathfrak{p},\beta)\frac{d\zeta}{\zeta} -\int^z_{\infty}
(F(\mathfrak{p},\beta)^2+G(\mathfrak{p},\beta) )\frac{d\zeta}{\zeta^2} +\int^z_{\infty}
\frac{\chi-\beta'h}{2P(\mathfrak{p})} \frac{d\zeta}{\zeta} +O(z^{-\mu-\delta}), 
\\
\chi- &\beta'h= \frac 65A_{\phi}h +\frac 15 \int^z_{\infty} H(\mathfrak{p},\beta)\frac
{d\zeta}{\zeta} +\frac 15 \int^z_{\infty} R(h,\chi)\frac {d\zeta}{\zeta}
+O(z^{-2\mu}),
\end{align*}
in which $H(\mathfrak{p},\beta)=8e^{i\phi} F(\mathfrak{p},\beta)\mathfrak{p}-\beta,$ every integral
converges, and
\begin{align*}
&\int^z_{\infty} F(\mathfrak{p},\beta)\frac{d\zeta}{\zeta} \ll z^{-\delta}, \quad
\int^z_{\infty}( F(\mathfrak{p},\beta)^2 +G(\mathfrak{p},\beta)) \frac{d\zeta}{\zeta^2} 
\ll z^{-1}, \quad \int^z_{\infty}\frac{\chi-\beta'h}{2P(\mathfrak{p})}\frac{d\zeta}
{\zeta}\ll z^{-\delta},
\\
&\int^z_{\infty} H(\mathfrak{p},\beta)\frac{d\zeta}{\zeta} \ll z^{-\delta}, \quad
\int^z_{\infty} R(h,\chi) \frac{d\zeta}{\zeta} \ll z^{-\delta}. 
\end{align*}
\end{prop}
\section{Proofs of the main results}\label{sc4}
By Proposition \ref{prop3.3},
$$
h-\frac 35A_{\phi}\int^z_{\infty} \frac h{P(\mathfrak{p})}\frac{d\zeta}{\zeta}
=\int^z_{\infty}F(\mathfrak{p},\beta)\frac{d\zeta}{\zeta} -\int^z_{\infty}(F(\mathfrak{p},\beta)
^2+G(\mathfrak{p},\beta))\frac{d\zeta}{\zeta^2} +I_1+O(z^{-\mu-\delta})
$$
with
$$
I_1=\frac 1{10}\int^z_{\infty}\frac 1{P(\mathfrak{p})} \int^{\zeta}_{\infty}(H(\mathfrak{p},
\beta)+R(h,\chi) )\frac{d\zeta_1}{\zeta_1} \frac{d\zeta}{\zeta}
$$
in $\check{\mathcal{D}}_{\mathrm{cut}}(\phi,t_{\infty},c,\varepsilon)$ 
since $\mu\ge\delta.$ Note that 
$\int^z_{\infty} P(\mathfrak{p})^{-1}\zeta^{-1}d\zeta \ll z^{-1}$ by
Lemma \ref{lem5.5} shown later, 
and $\int^z_{\infty}H(\mathfrak{p},\beta)\zeta^{-1}d\zeta \ll
z^{-\delta}$, $\int^z_{\infty} R(h,\chi) \zeta^{-1}d\zeta \ll z^{-\delta}$ by
Proposition \ref{prop3.3} and \eqref{3.4}. Then integration by parts leads to
\begin{align}\notag
10 I_1 &= \int^z_{\infty}\frac 1{P(\mathfrak{p})} \frac{d\zeta}{\zeta} \int^z_{\infty}
(H(\mathfrak{p},\beta)+R(h,\chi))\frac{d\zeta}{\zeta} 
\\
\notag
&\phantom{---} -\int^z_{\infty}\int^{\zeta}
_{\infty}\frac 1{P(\mathfrak{p})} \frac{d\zeta_1}{\zeta_1}(H(\mathfrak{p},\beta)+R(h,\chi))
\frac{d\zeta}{\zeta}
\\
\label{4.a}
&= -\int^z_{\infty} H(\mathfrak{p},\beta)\int^{\zeta}_{\infty}\frac 1{P(\mathfrak{p})}\frac
{d\zeta_1}{\zeta_1} \frac{d\zeta}{\zeta}+O(z^{-1-\delta}).
\end{align}
Furthermore, by \eqref{3.3}, \eqref{3.5}, Lemma \ref{lem5.5} and Proposition
\ref{prop3.2},
\begin{align}\notag
\int^z_{\infty}\frac{h}{P(\mathfrak{p})} \frac{d\zeta}{\zeta} &= h\int^z_{\infty}
\frac 1{P(\mathfrak{p})}\frac{\d\zeta}{\zeta}-\int^z_{\infty} h'\int^{\zeta}_{\infty}
\frac 1{P(\mathfrak{p})} \frac{d\zeta_1}{\zeta_1} d\zeta
\\
\notag
&= -\int^z_{\infty}\Bigl(F(\mathfrak{p},\beta)+\frac{\chi}{2P(\mathfrak{p})}\Bigr) \int^{\zeta}
_{\infty} \frac 1{P(\mathfrak{p})}\frac{d\zeta_1}{\zeta_1}\frac{d\zeta}{\zeta}
+O(z^{-1-\mu})
\\
\label{4.b}
&= -\int^z_{\infty}F(\mathfrak{p},\beta) \int^{\zeta}_{\infty}\frac 1{P(\mathfrak{p})} 
\frac{d\zeta_1}{\zeta_1}\frac{d\zeta}{\zeta} +O(z^{-1-\delta}). 
\end{align}
Using \eqref{4.a} and \eqref{4.b} we obtain, under supposition \eqref{3.5},
\begin{align}\notag
h=\int^z_{\infty}& F(\mathfrak{p},\beta)\frac{d\zeta}{\zeta} -\int^z_{\infty}(F(\mathfrak{p},\beta)
^2+G(\mathfrak{p},\beta)) \frac{d\zeta}{\zeta^2}
\\
\label{4.1}
& -\frac 1{10}\int^z_{\infty}K(\mathfrak{p},\beta)
\mathcal{I}(\zeta)\frac{d\zeta}{\zeta} +O(z^{-\mu-\delta})
\end{align}
in $\check{\mathcal{D}}_{\mathrm{cut}}(\phi,t_{\infty},c,\varepsilon),$ 
in which the
implied constant possibly depends on $\mu$, and $K(\mathfrak{p},\beta)=H(\mathfrak{p},\beta)
+6A_{\phi}F(\mathfrak{p},\beta)=2(4e^{i\phi}\mathfrak{p}+3A_{\phi})F(\mathfrak{p},\beta)-\beta,$
$\mathcal{I}(z)=\int^z_{\infty} P(\mathfrak{p})^{-1}\zeta^{-1}d\zeta.$  
The integrals on the right-hand side of \eqref{4.1} satisfy
\begin{equation}\label{4.2}
\begin{split}
&\int^z_{\infty}F(\mathfrak{p},\beta)\frac{d\zeta}{\zeta} \ll z^{-1}, \quad
\int^z_{\infty}(F(\mathfrak{p},\beta)^2 +G(\mathfrak{p},\beta))\frac{d\zeta}{\zeta^2}
 \ll z^{-1}, 
\\
&\int^z_{\infty}K(\mathfrak{p},\beta)\mathcal{I}(\zeta)\frac{d\zeta}{\zeta}
 \ll z^{-1}, \quad \mathcal{I}(z)\ll z^{-1},
\end{split}
\end{equation}
in which the first nontrivial estimate follows from Lemmas \ref{lem5.5} 
and \ref{lem5.4} with 
$\mathfrak{p}/\sqrt{P(\mathfrak{p})}=\mathfrak{p}\mathfrak{p}'/P(\mathfrak{p})$ 
and $\beta(z)=\tfrac 85 e^{i\phi}\omega_{\mathbf{a}}^{-1} (\beta_0 +
\mathfrak{b}(z-t_0))$ (cf. \eqref{5.6}).
\subsection{Derivation of Theorem \ref{thm2.2} and Corollary \ref{cor2.3}}
\label{ssc4.1}
To prove Theorem \ref{thm2.2} let us start with \eqref{1.3} and \eqref{1.5}
of Theorems A and B with given $c>0,$ $\varepsilon >0$ and a small $\delta>0.$
By \eqref{4.2} the sum of the integrals on the right-hand side of \eqref{4.1}
is $O(z^{-1})$ in $\check{\mathcal{D}}_{\mathrm{cut}}(\phi,t_{\infty},c,
\varepsilon)$ for sufficiently large $t_{\infty}$.
Note that \eqref{3.5} is valid for $\mu=\delta,$ i.e. $h(z)\ll z^{-\delta}.$
Then, by \eqref{4.1} with $\mu=\delta,$ we have $h(z)\ll |z^{-1}|
+|z^{-2\delta}|.$ If $2\delta <1,$ then \eqref{3.5} is valid for $\mu=2\delta,$
which implies asymptotic formula \eqref{4.1} with the error bound 
$O(z^{-3\delta})$ and $h(z) \ll |z^{-1}|+|z^{-3\delta}|.$ 
For $m_0$ such that $m_0\delta <1 \le (m_0+1)\delta,$ $m_0$-times
repetition of this procedure results in
\eqref{4.1} with $\mu=m_0\delta$, which yields \eqref{3.5} with $\mu=1.$  
Then \eqref{4.1} with $\mu=1$ is valid in 
$\check{\mathcal{D}}_{\mathrm{cut}}(\phi,t_{\infty},c,\varepsilon)$, implying 
Theorem \ref{thm2.2}. Calculation of
the coefficient of $\beta_0^2$ by the use of Lemma \ref{lem5.5} leads us to
Corollary \ref{cor2.3}.
\subsection{Derivation of Theorem \ref{thm2.1}}\label{ssc4.2}
Theorem \ref{thm2.2} implies $h(z)\ll z^{-1}$ in $\check{\mathcal{D}}_{\mathrm
{cut}}(\phi,t_{\infty},c,\varepsilon)$. Note that
\begin{align*}
& (e^{-i\phi}x)^{-1/2}y(x)-\wp(e^{i\phi}t-t_0; g_2(\phi),g_3(\phi))
\\
=& \wp(e^{i\phi}t-t_0 +h(e^{i\phi}t); g_2(\phi),g_3(\phi))
-\wp(e^{i\phi}t-t_0; g_2(\phi),g_3(\phi))
\\
\ll & \wp'(e^{i\phi}t-t_0; g_2(\phi),g_3(\phi)) h(e^{i\phi}t) \ll t^{-1}
\end{align*}
in $\check{\mathcal{D}}_{\mathrm{cut}}(\phi,t_{\infty},c,\varepsilon)$, 
and that the function on the first line is holomorphic in  
${\mathcal{D}}(\phi,t_{\infty},c,\varepsilon)$. By using the maximal 
modulus principle
in excluded discs around $t_j^{m,n}$ $(1\le j \le 3;\, m,n\in\mathbb{Z}),$ 
we obtain Theorem \ref{thm2.1}.
\section{Lemmas on primitive functions}\label{sc5}
It remains to show the lemmas on primitive functions used in the proofs of
the main results. Recall that $P(\lambda)=4w(A_{\phi},\lambda)^2$ as in
Theorem \ref{thm2.2}, and let us write
\begin{align*}
P(\lambda)&=4\lambda^3-g_2\lambda -g_3 =4(\lambda-\lambda_1)(\lambda-\lambda_2)
(\lambda-\lambda_3), \quad g_2=-2e^{i\phi}, \quad g_3=-A_{\phi},
\\
&\lambda_1=\wp(\omega_1),\quad  \lambda_2=\wp(\omega_2),\quad 
\lambda_3=\wp(\omega_3), 
\end{align*}
where 
$$
\omega_1=\frac 12\omega_{\mathbf{a}},\quad \omega_3=\frac 12
\omega_{\mathbf{b}}, \quad \omega_1+\omega_2+\omega_3=0.
$$
Then
\begin{align*}
&\lambda_1+\lambda_2+\lambda_3=0, \quad \lambda_2\lambda_3+\lambda_3\lambda_1
+\lambda_1\lambda_2=-g_2/4, \quad \lambda_1\lambda_2\lambda_3=g_3/4,
\\
& (\lambda_2-\lambda_3)^2(\lambda_3-\lambda_1)^2(\lambda_1-\lambda_2)^2
=\frac 1{16}(g_2^3-27g_3^2),
\end{align*}
and $\omega_{\mathbf{a}}$ and $J_{\mathbf{a}}$ are also written in the form
$$
\omega_{\mathbf{a}}=\int_{\mathbf{a}}
 \frac{d\lambda}{\sqrt{P(\lambda)}}, \quad
J_{\mathbf{a}}=\int_{\mathbf{a}}
 \sqrt{P(\lambda)} d\lambda.
$$
\begin{lem}\label{lem5.1}
Let 
\begin{align*}
&\gamma_1=(\lambda_1-\lambda_2)(\lambda_1-\lambda_3), \quad
\gamma_2=(\lambda_2-\lambda_3)(\lambda_2-\lambda_1), \quad
\gamma_3=(\lambda_3-\lambda_1)(\lambda_3-\lambda_2), 
\\
&\gamma_0=5(g_2^3-27g_3^2)^{-1}
\end{align*}
and $\wp(z)=\wp(z;g_2, g_3).$ Then
\begin{align*}
\int^z_0  \frac{dz}{P(\wp(z))} =& -\frac 1{4\omega_{\mathbf{a}}} \Bigl(
\gamma_1^{-2}\Bigl(\frac{5J_{\mathbf{a}}}{2g_2}z+\frac{\vartheta'}{\vartheta}
\Bigl(\frac z{\omega_{\mathbf{a}}} +\frac{\tau}2,\tau\Bigr)\Bigr)
+\gamma_2^{-2}\Bigl(\frac{5J_{\mathbf{a}}}{2g_2}z+\frac{\vartheta'}{\vartheta}
\Bigl(\frac z{\omega_{\mathbf{a}}} ,\tau\Bigr)\Bigr)
\\
&+\gamma_3^{-2}\Bigl(\frac{5J_{\mathbf{a}}}{2g_2}z+\frac{\vartheta'}{\vartheta}
\Bigl(\frac z{\omega_{\mathbf{a}}} +\frac{1}2,\tau\Bigr)\Bigr) \Bigr) +C_0,
\\
\int^z_0  \frac{dz}{P(\wp(z))^2} =& \gamma_0z 
-\frac 1{96\omega_{\mathbf{a}}} \Bigl(
\gamma_1^{-4} (\partial_z^2 -48\lambda_1)
\Bigl(\frac{5J_{\mathbf{a}}}{2g_2}z+\frac{\vartheta'}{\vartheta}
\Bigl(\frac z{\omega_{\mathbf{a}}} +\frac{\tau}2,\tau\Bigr)\Bigr)
\\
&+\gamma_2^{-4} (\partial_z^2 -48\lambda_2)
\Bigl(\frac{5J_{\mathbf{a}}}{2g_2}z+\frac{\vartheta'}{\vartheta}
\Bigl(\frac z{\omega_{\mathbf{a}}} ,\tau\Bigr)\Bigr)
\\
&+\gamma_3^{-4} (\partial_z^2 -48\lambda_3)
\Bigl(\frac{5J_{\mathbf{a}}}{2g_2}z+\frac{\vartheta'}{\vartheta}
\Bigl(\frac z{\omega_{\mathbf{a}}} +\frac{1}2,\tau\Bigr)\Bigr) \Bigr) +C_1,
\end{align*}
where $C_0$ and $C_1$ are some constants and $\partial_z=d/dz$.
\end{lem}
\begin{proof}
Around $\omega_j$ $(j=1,2,3)$,
$$
\frac 1{P(\wp(z))}=\frac 14\gamma_j^{-2}(z-\omega_j)^{-2}(1+O(z-\omega_j)^2)
$$
since $\wp(z)=\lambda_j+\gamma_j(z-\omega_j)^2+O(z-\omega_j)^4.$ Hence
\begin{equation}\label{5.1}
\frac 1{P(\wp(z))}-\frac 14(\gamma_1^{-2}\wp(z-\omega_1)+
\gamma_2^{-2}\wp(z-\omega_2)+ \gamma_3^{-2}\wp(z-\omega_3) )\equiv \Gamma_0,
\end{equation}
in which, by putting $z=0$, we find $\Gamma_0= 9g_3 (g_2^3-27g_3^2)^{-1}.$
We may set, for some $c_0$,
\begin{equation}\label{5.2}
\wp(z)+\frac 1{\omega_{\mathbf{a}}}\partial_z \frac{\vartheta'}{\vartheta}
\Bigl(\frac z{\omega_{\mathbf{a}}}+\nu, \tau\Bigr)\equiv c_0,
\end{equation}
integration of which yields
$$
\omega_{\mathbf{a}} c_0 =c_0 \int^{\omega_1}_{-\omega_1}dz
= \int^{\omega_1}_{-\omega_1}\wp(z) dz =\int_{\mathbf{a}}\frac{\lambda}
{\sqrt{P(\lambda)}}d\lambda. 
$$
Observing that
\begin{align*}
J_{\mathbf{a}}= &\int_{\mathbf{a}}\sqrt{P(\lambda)}d\lambda
=\int_{\mathbf{a}}\Bigl(\frac 23\lambda\, \bigl(\!\sqrt{P(\lambda)}\bigr)' 
-\frac 23 
\frac{g_2\lambda}{\sqrt{P(\lambda)}} -\frac{g_3}{\sqrt{P(\lambda)}}\Bigr)d\lambda
\\
=& -\frac 23 J_{\mathbf{a}}-\frac 23 g_2 \int_{\mathbf{a}} \frac{\lambda}
{\sqrt{P(\lambda)}} d\lambda -g_3 \omega_{\mathbf{a}},
\end{align*}
we have $c_0=-5J_{\mathbf{a}}/(2g_2\omega_{\mathbf{a}})-3g_3/(2g_2).$
Inserting \eqref{5.2} with $z\mapsto z-\omega_j$ $(j=1,2,3)$ into \eqref{5.1}, 
and using
$$
\Gamma_0-\frac{3g_3}{8g_2}(\gamma_1^{-2}+\gamma_2^{-2}+\gamma_3^{-2})=0,
$$
we obtain the first primitive function.
To derive the second formula, we note that, around $z=\omega_j$ $(j=1,2,3)$,
\begin{align*}
\frac{16}{P(\wp(z))^2}=&
\Bigl(\gamma_j^{-2}\wp(z-\omega_j)-4\gamma_j^{-2}\lambda_j+O(z-\omega_j)^2
\Bigr)^2 
\\
=&\gamma_j^{-4}(\wp(z-\omega_j)^2-8\lambda_j\wp(z-\omega_j)+O(1))
\end{align*}
by \eqref{5.1}, since, say around $z=\omega_1$, 
$\wp(z-\omega_2)=\wp(\omega_3)+O(z-\omega_1)^2,$ 
$\wp(z-\omega_3)=\wp(\omega_2)+O(z-\omega_1)^2,$ with $\gamma_2^{-2}\lambda_3
+\gamma_3^{-2}\lambda_2+4\Gamma_0= -4\gamma_1^{-2}\lambda_1.$  
Then we set
\begin{align}\notag
\frac{16}{P(\wp(z))^2} &-\gamma_1^{-4} (\wp(z-\omega_1)^2-8\lambda_1
\wp(z-\omega_1) )  -\gamma_2^{-4} (\wp(z-\omega_2)^2-8\lambda_2\wp(z-\omega_2) )
\\
\label{5.3}
 &-\gamma_3^{-4} (\wp(z-\omega_3)^2-8\lambda_3\wp(z-\omega_3) )\equiv \Gamma_1
\end{align}
with
$$
\Gamma_1=\frac{7((\lambda_2-\lambda_3)^4\lambda_1^2 +
(\lambda_3-\lambda_1)^4\lambda_2^2+(\lambda_1-\lambda_2)^4\lambda_3^2)}
{(\lambda_2-\lambda_3)^{4}(\lambda_3-\lambda_1)^{4}
(\lambda_1-\lambda_2)^{4}}.
$$
Insertion of \eqref{5.2} and
$$
\wp(z)^2+ \frac {\partial_z^3}{6\omega_{\mathbf{a}}}  \frac{\vartheta'}
{\vartheta}\Bigl(\frac{z}{\omega_{\mathbf{a}}}+\nu, \tau \Bigr)\equiv
\frac{g_2 }{12}
$$
with $z \mapsto z-\omega_j$ $(j=1,2,3)$ into \eqref{5.3} leads to the second
formula.
\end{proof}
\begin{lem}\label{lem5.2}
Under the same supposition as in Lemma \ref{lem5.1}, for $k=0,1,2$,
\begin{align*}
\int^z_{z_0} \frac{\wp(z)^k\wp'(z)}{P(\wp(z))} dz =&\frac 14 \Bigl(
\gamma_1^{-1}\lambda_1^k \ln(\wp(z)-\lambda_1)
\\
& +\gamma_2^{-1}\lambda_2^k \ln(\wp(z)-\lambda_2)+
\gamma_3^{-1}\lambda_3^k \ln(\wp(z)-\lambda_3)\Bigr) +C(z_0),
\end{align*}
where $z_0\not\in \{0,\omega_1,\omega_2,\omega_3\}+\omega_{\mathbf{a}}\mathbb{Z}
+\omega_{\mathbf{b}}\mathbb{Z},$ and $C(z_0)$ is some constant.
\end{lem}
\begin{proof}
Note that, for $k=0,1,2,$
$$
\frac{4\wp^k}{P(\wp)}=\gamma_1^{-1}\lambda_1^k(\wp-\lambda_1)^{-1}+
\gamma_2^{-1}\lambda_2^k(\wp-\lambda_2)^{-1}+
\gamma_3^{-1}\lambda_3^k(\wp-\lambda_3)^{-1},
$$
from which the lemma follows.
\end{proof}
Recall that $\mathfrak{p}(z)=\wp(z-t_0)$ and
\begin{equation}\label{5.6}
\beta(z)=\frac{8e^{i\phi}}{5\omega_{\mathbf{a}}}(\beta_0+\mathfrak{b}(z-t_0)),
\quad \mathfrak{b}(\sigma)=-\frac 54 e^{-i\phi}J_{\mathbf{a}}\sigma 
+\frac{\vartheta'}{\vartheta}\Bigl(\frac{\sigma}{\omega_{\mathbf{a}}} +\nu,
\tau\Bigr).
\end{equation}
Then we have the following.
\begin{lem}\label{lem5.3}
Let $\alpha_0 \in \mathbb{C}$ be a given number. Then
$$
\int^s_{\infty} (\mathfrak{b}(\sigma-\alpha_0)+\mathfrak{b}(\sigma+\alpha_0)
)_{\sigma}\mathfrak{b}(\sigma) \frac{d\sigma}{\tilde{\sigma}} \ll s^{-1},
\quad \tilde{\sigma}=\sigma+t_0
$$
as $x=e^{i\phi} (\tfrac 54e^{-i\phi}s)^{4/5} \to \infty$ through $\check
{\mathcal{D}}_{\mathrm{cut}}(\phi,t_{\infty},c,\varepsilon)$, in which 
the integral on the left-hand side is convergent.
\end{lem}
\begin{proof}
Let $x_{\nu}=e^{i\phi}(\tfrac 54 e^{-i\phi}s_{\nu})^{4/5}$ be any sequence
tending to $\infty$ through $\check{\mathcal{D}}_{\mathrm{cut}}(\phi,t_{\infty}
,c,\varepsilon)$. Note that $\mathfrak{b}(s)$ and $\mathfrak{b}(s\pm \alpha_0)$
are bounded in 
$\check{\mathcal{D}}_{\mathrm{cut}}(\phi,t_{\infty},c,\varepsilon)$.
Integration by parts leads to
\begin{align*}
\int^s_{s_{\nu}}\mathfrak{b}(\sigma-\alpha_0)_{\sigma} 
\mathfrak{b}(\sigma)\frac{d\sigma}{\tilde{\sigma}}
=& \Bigl[ \mathfrak{b}(\sigma-\alpha_0)\mathfrak{b}(\sigma)\tilde{\sigma}^{-1}
\Bigr]^s_{s_{\nu}}
\\
&-\int^s_{s_{\nu}} \mathfrak{b}(\sigma-\alpha_0)\mathfrak{b}_{\sigma}(\sigma)
\frac{d\sigma}{\tilde{\sigma}}
+\int^s_{s_{\nu}} \mathfrak{b}(\sigma-\alpha_0)\mathfrak{b}(\sigma)
\frac{d\sigma}{\tilde{\sigma}^2}
\\
=&-\int^{s-\alpha_0}_{s_{\nu}-\alpha_0} \mathfrak{b}(\sigma)
\mathfrak{b}_{\sigma}(\sigma+\alpha_0)
\frac{d\sigma}{\tilde{\sigma}+\alpha_0} +O(s^{-1})+O(s_{\nu}^{-1})
\\
=&-\int^{s}_{s_{\nu}}
\mathfrak{b}(\sigma+\alpha_0)_{\sigma} \mathfrak{b}(\sigma)
\frac{d\sigma}{\tilde{\sigma}} +O(s^{-1})+O(s_{\nu}^{-1}).
\end{align*}
The passage to the limit $s_{\nu} \to \infty$ leads to the lemma.
\end{proof}
By Lemma \ref{lem5.1} with $g_2=-2e^{i\phi},$ we may write
\begin{align*}
\frac 1{P(\wp(z))} =& -\frac 1{4\omega_{\mathbf{a}}} \frac d{dz}
\Bigl( \frac{\gamma_1^{-2}}2 (\mathfrak{b}(z-\omega_{\mathbf{a}}/2)
 +\mathfrak{b}(z+\omega_{\mathbf{a}}/2))
\\
&+ \frac{\gamma_2^{-2}}2 (\mathfrak{b}(z-\omega_{\mathbf{a}}\nu)
 +\mathfrak{b}(z+\omega_{\mathbf{a}}\nu))
+ \frac{\gamma_3^{-2}}2 (\mathfrak{b}(z-\omega_{\mathbf{a}}\tau/2)
 +\mathfrak{b}(z+\omega_{\mathbf{a}}\tau/2)) \Bigr).
\end{align*}
Substituting $z$ by $z-t_0$ and using Lemma 5.3 we have the following.
\begin{lem}\label{lem5.4}
In $\check{\mathcal{D}}_{\mathrm{cut}}(\phi,t_{\infty},c,\varepsilon)$,
$$
\int^z_{\infty} \frac{\mathfrak{b}(\zeta-t_0)}{P(\mathfrak{p}(\zeta))} \frac{d\zeta}
{\zeta} \ll z^{-1},
$$
in which the integral on the left-hand is convergent.
\end{lem}
\begin{lem}\label{lem5.5}
In $\check{\mathcal{D}}_{\mathrm{cut}}(\phi,t_{\infty},c,\varepsilon)$,
$$
\int^z_{\infty} \frac 1{P(\mathfrak{p}(\zeta))} \frac{d\zeta}{\zeta} \ll z^{-1},
\,\,\,
\int^z_{\infty} \frac 1{P(\mathfrak{p}(\zeta))^2} \frac{d\zeta}{\zeta^2} =-\gamma_0 z^{-1}
+O(z^{-2}),
\,\,\,
\int^z_{\infty} \frac{\mathfrak{p}(\zeta)\mathfrak{p}'(\zeta)}{P(\mathfrak{p}(\zeta))} \frac{d\zeta}
{{\zeta}} \ll z^{-1},
$$
where $\gamma_0=-5(8e^{3i\phi}+27A_{\phi}^2)^{-1}.$
\end{lem}
\begin{proof}
Let $\gamma_0z+F(z)$ be the primitive function of $1/P(\wp(z))^2$ given in
Lemma \ref{lem5.1}, where $F(z)$ is bounded in $\check{\mathcal{D}}_{\mathrm
{cut}}(\phi,t_{\infty},c,\varepsilon)$. Then
\begin{align*}
\int^z_{\infty} \frac 1{P(\mathfrak{p}(\zeta))^2} \frac{d\zeta}{{\zeta}^2}
=& \Bigl[ (F(\zeta-t_0)+\gamma_0\zeta)\frac 1{{\zeta}^{2}} \Bigl]^z_{\infty}
+2\int^z_{\infty}(F(\zeta-t_0)+\gamma_0\zeta) \frac{d\zeta}{{\zeta}^3}
\\
=& -\gamma_0 z^{-1}+ O(z^{-2}),
\end{align*}
which is the second integral. The remaining estimates are similarly obtained
by the use of Lemmas \ref{lem5.1} and \ref{5.2}.
\end{proof}
\section{Discussion}\label{sc6}
Our main results Theorems \ref{thm2.1} and \ref{thm2.2} are derived
by combining expressions \eqref{1.3} and \eqref{1.5} of $y(x)$ and of 
$B(\phi,t)$ in Theorems A and B, respectively,
with analysis on the integral equations (associated with system \eqref{3.1}) 
in Proposition \ref{prop3.3} via \eqref{4.1}. 
The formulas of Theorems A and B with the error bounds $O(t^{-\delta})$ 
have been obtained in \cite{K1} and \cite{K-3}
by WKB analysis on isomonodromy linear system \eqref{1.2}. 
Our approach to the error term is said to be purely classical iterative 
calculation via \eqref{4.1} by the use of \eqref{1.3} and \eqref{1.5}.
This is just in the situation as referred to in \cite[Remark 1]{K-3}.  
The essential part of this procedure owes much to the structure of the
integral relation \eqref{4.1} containing the term 
$$
\int^z_{\infty} F(\mathfrak{p},\beta) \frac{d\zeta}{\zeta} \ll z^{-1},
$$  
in which the key to the estimate is Lemma \ref{lem5.3}.  
The corresponding error $O(t^{-1})$, say in \cite{I-K}, is a quantity arising in 
constructing the solution of the related Riemann-Hilbert problem. 
\par
It may be conjectured that the error bound of \eqref{1.5} in Theorem B 
is $O(t^{-1})$. If this is valid, then, in Theorem \ref{thm2.2}, the error
bound of the expression for $h(z)$ is $O(z^{-2}).$ 
For the
fifth Painlev\'e transcendents the corresponding error bound estimates
are valid \cite{S1}, \cite{S}.


\end{document}